\documentclass{amsart}
\usepackage{graphicx}
\usepackage[all]{xy}
\usepackage[ansinew]{inputenc}
\usepackage{amsmath}
\usepackage{amscd}
\usepackage{amssymb}
\usepackage{latexsym}
\usepackage[active]{srcltx}
\usepackage{mathrsfs}
\usepackage{color}

\newtheorem{thm}{Theorem}[section]

\theoremstyle{definition}
\newtheorem{cor}[thm]{Corollary}
\newtheorem{lem}[thm]{Lemma}
\newtheorem{prop}[thm]{Proposition}
\newtheorem{defn}[thm]{Definition}

\newtheorem{fact}[thm]{Fact}

\newtheorem*{thmA}{Theorem A}
\newtheorem*{thmB}{Theorem B}
\newtheorem*{thmC}{Theorem C}
\newtheorem*{thmD}{Theorem D}

\numberwithin{equation}{section}

\newcommand{\N}{\mathbb{N}}
\newcommand{\Z}{\mathbb{Z}}
\newcommand{\Q}{\mathbb{Q}}
\newcommand{\R}{\mathbb{R}}

\def \gcd {\operatorname{gcd}}

\newcommand{\Cal}{\mathcal}

\def \<{\langle}
\def \>{\rangle}

\def \((  {(\!(}
\def \)) {)\!)}

\begin{document}

\title[]
{When is scalar multiplication decidable?}

\author[P. Hieronymi]{Philipp Hieronymi}
\address
{Department of Mathematics\\University of Illinois at Urbana-Champaign\\1409 West Green Street\\Urbana, IL 61801}
\email{phierony@illinois.edu}
\urladdr{http://www.math.uiuc.edu/\textasciitilde phierony}


\thanks{The author was partially supported by NSF grant DMS-1300402 and by UIUC Campus Research Board award 14194.}
\date{\today}

\begin{abstract} Let $K$ be a subfield of $\R$. The theory of $\R$ viewed as an ordered $K$-vector space and expanded by a predicate for $\Z$ is decidable if and only if $K$ is a real quadratic field.
\end{abstract}
\maketitle

\section{Introduction}

It has long been known that the first order theory of the structure $(\R,<,+,\Z)$ is decidable. Arguably due to Skolem \cite{skolem}\footnote{See Smorny\'nski \cite[Exercise III.4.15]{smor}.}, the result can be deduced easily from B\"uchi's theorem on the decidability of monadic second order theory of one successor \cite{Buchi}\footnote{See Boigelot, Rassart and Wolper \cite{BRW}.}, and was later rediscovered independently by Weispfenning \cite{weis} and Miller \cite{ivp}. However, a consequence of G\"odel's famous first incompleteness theorem \cite{Goedel} states that when expanding  $(\R,<,+,\Z)$  by a symbol for multiplication on $\R$, the theory of the resulting structure $(\R,<,+,\cdot,\Z)$ becomes undecidable. This observation gives rise to the following natural and surprisingly still open question:

\begin{center}\emph{How many traces of multiplication can be added to $(\R,<,+,\Z)$ without making the first order theory undecidable?}
\end{center}

\noindent Here, building on earlier work of Hieronymi and Tychonievich \cite{HT} and Hieronymi \cite{H-Twosubgroups}, we will give a complete answer to this question when \emph{traces of multiplication} is taken to mean  scalar multiplication by certain irrational numbers. To make this statement precise: for $a\in \R$, let $\lambda_a:\R \to \R$ be the function that takes $x$ to $ax$. Denote the structure $(\R,<,+,\Z,\lambda_a)$ by $\Cal S_a$. A real number is quadratic if it is the solution to a quadratic equation with rational coefficients. Theorem B of \cite{HT} states that the theory of $\Cal S_a$ is undecidable if $a$ is not quadratic. In this paper we prove that $\Cal S_a$ is decidable if $a$ is quadratic. This establishes the following theorem.

\begin{thmA} The theory of $\Cal S_a$ is decidable if and only if $a$ is quadratic.
\end{thmA}

\noindent By Theorem A of \cite{H-Twosubgroups}, the theory of the structure $(\R,<,+,\Z,\Z a)$ is decidable whenever $a$ is quadratic. Here, we will show how the decidablity of the theory of $\Cal S_a$ can be deduced from this result. Before explaining the precise strategy of the proof, we collect two corollaries of Theorem A.\newline

\noindent Theorem A induces a dichotomy for expansions of $(\R,<,+,\Z)$ by scalar multiplication by a single real number. This raises the question whether there is a similar characterization for expansions where scalar multiplication is added for every element of some subset of $\R$.
We say that real numbers $a_1,\dots,a_n$ are $\Q$-linear dependent if there are $q_1,\dots, q_n\in \Q$, not all zero, such that
\[
q_1 a_1 + \dots + q_n a_n = 0.
\]
We say $a_1,\dots, a_n$ are $\Q$-linear independent if they are not $\Q$-linear dependent. Theorem C of \cite{HT} states that the structure $(\R,<,+,\Z,\Z a,\Z b)$ defines full multiplication on $\R$ whenever $1,a,b$ are $\Q$-linear independent. Since $(\R,<,+,\Z,\lambda_a,\lambda_b)$ defines both $\Z a$ and $\Z b$, it also defines full multiplication on $\R$ whenever $1,a,b$ are $\Q$-linear independent. On the other hand, if $a,b$ are irrational and $1,a,b$ are $\Q$-linear dependent, then either $\Cal S_{a}$ defines the function $\lambda_b$. With this observation we get the following result as a corollary of Theorem A and \cite[Theorem B]{HT}.

\begin{thmB} Let $S\subseteq \R$. Then the structure $(\R,<,+,\Z,(\lambda_b)_{b\in S})$ defines the same sets as exactly one of the following structures:
\begin{itemize}
\item[(i)] $(\R,<,+,\Z)$,
\item[(ii)] $(\R,<,+,\Z,\lambda_a)$, for some quadratic $a\in \R\setminus \Q$,
\item[(iii)] $(\R,<,+,\cdot,\Z)$.
\end{itemize}
\end{thmB}

\noindent The three cases are indeed exclusive. By Theorem A, a structure in (ii) does not define full multiplication on $\R$. Using the results in \cite{weis} or \cite{ivp}, one can show that $(\R,<,+,\Z)$ does not define any dense and codense subset of $\R$, while the structures in (ii) do\footnote{For a definable dense set in $\Cal S_a$, see \cite[Proof of Theorem C]{HT}.}. As a corollary of Theorem B, we obtain the following generalization of Theorem A.

\begin{thmC} Let $K$ be a subfield of $\R$.  The theory of the ordered $K$-vector space $\R$ expanded by a predicate for $\Z$ is decidable if and only if $K$ is a quadratic field.
\end{thmC}

\noindent The work in this paper is mainly motivated by purely foundational concerns. However, the structure $(\R,<,+,\Z)$ and the decidability of its first order theory have been used extensively in computer science, in particular in verification and model checking. Our Theorem A gives decidability in a larger language, and one might hope that the increase in expressive power leads to new applications; see Hieronymi, Nguyen, Pak \cite{HNP} for an application to decision problems in discrete geometry. One should mention however that for irrational $a$, the structure $\Cal S_a$ defines a model of the monadic second order theory of one successor by \cite[Theorem D]{H-Twosubgroups}. Thus any implementation of the algorithm determining the truth of a sentence in $\Cal S_a$ is limited by the high computational costs necessary to decide a statement in the monadic second order theory of one successor\footnote{Most of the computational complexity comes from the construction of the complement of a B\"uchi automaton. For details, see for example Vardi \cite{Vardi}. When considering just $(\R,<,+,\Z)$, some of the difficulties can be avoided, see Boigelot, Jodogne and Wolper \cite{BJW}.}.   \\

\noindent We have already argued how Theorem A implies Theorem B and Theorem C. Theorem A itself is deduced from the following result.

\begin{thmD} Let $d\in \Q$. Then $(\R,<,+,\Z,\Z\sqrt{d})$ defines multiplication by $\sqrt{d}$.
\end{thmD}

\noindent The proof of Theorem D is the goal of this paper and the only significant improvement over previous results. We now explain how Theorem A can be proved using Theorem D.

\begin{proof}[Proof of Theorem A from Theorem D]
By \cite[Theorem B]{HT}, the theory of $\Cal S_a$ is undecidable whenever $a$ is not quadratic. To establish Theorem A, it is therefore enough to show that the theory of $\Cal S_a$ is decidable for quadratic $a$. 
Let $a\in \R$ be quadratic. Then there are $b,c,d\in \Q$ such that
$a= b + c\sqrt{d}$. By \cite[Theorem A]{H-Twosubgroups} the theory of $(\R,<,+,\Z,\Z\sqrt{d})$ is decidable. By Theorem D, the function $\lambda_{\sqrt{d}}$ is definable in $(\R,<,+,\Z,\Z\sqrt{d})$. Since $a=b+c\sqrt{d}$, $\lambda_a$ is definable in this structure as well. The decidability of the theory of $\mathcal S_a$ follows. 
\end{proof}

For ease of notation, we denote $(\R,<,+,\Z,\Z a)$ by $\Cal R_{a}$. Theorem D is not the first results of this form. Let $\varphi:=\frac{1+\sqrt{5}}{2}$ be the golden ratio. Then \cite[Theorem B]{H-Twosubgroups} states that $\Cal R_{\varphi}$ defines $\lambda_\varphi$. The proof of this result depends heavily on the fact that the continued fraction expansion of $\varphi$ is $[1;1,\dots]$. To prove Theorem D, we build on this earlier work in \cite{H-Twosubgroups}, but have to add extra arguments coming both from the theory of continued fractions and from definability. In Section 4 of \cite{H-Twosubgroups} it is shown that the representations in the Ostrowski numeration system based on $a$ of both natural numbers and real numbers are definable in $\Cal R_{\sqrt{d}}$. The Ostrowski numeration system is a non-standard way to represent numbers based on the continued fraction expansion of $a$. In Section 2 we recall the basic definitions and results about this numeration system. In Section 3, after reminding the reader of some of the previous results in \cite{H-Twosubgroups}, we prove that $\lambda_{\sqrt{d}}$ is definable in $\Cal R_{\sqrt{d}}$. The main step in the proof is to realize that using theorems about the continued fraction expansions of $\sqrt{d}$, multiplication by $\sqrt{d}$ can be expressed in terms of certain shifts in the Ostrowski representations and scalar multiplication by rational numbers. Most of the work in Section 3 will go into showing that these shifts are definable.

\subsection*{Acknowledgements} The author thanks Evgeny Gordon for interest in this work and helpful comments, Alexis Bès, Bernarnd Boigelot, Véronique Bruyère, Christian Michaux and Françoise Point for pointing out references, and the anonymous referee for carefully reading this paper.

\subsection*{Notation} We denote $\{0,1,2,\dots\}$ by $\N$. Throughout this paper definable will mean definable without parameters.

\section{Continued fractions} In this section, we recall some basic and well-known definitions and results about continued fractions. Expect for the definition of Ostrowski representations of real numbers, all these results can be found in every basic textbook on continued fractions. We refer the reader to Rockett and Szüsz \cite{RS} for proofs of the results, simply because to the author's knowledge it is the only book discussing Ostrowski representations of real numbers in detail. \newline

\noindent A \textbf{finite continued fraction expansion} $[a_0;a_1,\dots,a_k]$ is an expression of the form
\[
a_0 + \frac{1}{a_1 + \frac{1}{a_2+ \frac{1}{\ddots +  \frac{1}{a_k}}}}
\]
For a real number $a$, we say $[a_0;a_1,\dots,a_k,\dots]$ is \textbf{the continued fraction expansion of $a$} if $a=\lim_{k\to \infty}[a_0;a_1,\dots,a_k]$ and $a_0\in \Z$, $a_i\in \N_{>0}$ for $i>0$. For the rest of this subsection, fix a positive irrational real number $a$ and assume that $[a_0;a_1,\dots,a_k,\dots]$ is the continued fraction expansion of $a$.

\begin{defn}\label{def:beta} Let $k\geq 1$. We define $p_k/q_k \in \Q$ to be the \textbf{$k$-th convergent of $a$}, that is the quotient $p_k/q_k$ where $p_k\in \N$, $q_k\in \Z$, $\gcd(p_k,q_k)=1$ and 
\[
\frac{p_k}{q_k} = [a_0;a_1,\dots,a_k].
\]
The \textbf{$k$-th difference of $a$} is defined as $\beta_k := q_k a - p_k$. We define $\zeta_k \in \R$ to be the \textbf{$k$-th complete quotient of $a$}, that is
$\zeta_k = [a_k;a_{k+1},a_{k+2},\dots]$.
\end{defn}

\noindent Maybe the most important fact about the convergents we will use, is that both their nominators and denominators satisfy the following recurrence relation.

\begin{fact}{\cite[Chapter I.1 p. 2]{RS}}\label{fact:recursive} Let $q_{-1} := 0$ and $p_{-1}:=1$. Then $q_{0} = 1$, $p_{0}=a_0$ and for $k\geq 0$,
\begin{align*}
q_{k+1} &= a_{k+1} \cdot q_k + q_{k-1}, \\
p_{k+1} &= a_{k+1} \cdot p_k + p_{k-1}. \\
\end{align*}
\end{fact}
\noindent We directly get that for $k\geq 0$, $\beta_{k+1} = a_{k+1} \beta_k + \beta_{k-1}$. We need the following well-known facts about $\zeta_k$.

\begin{fact}{\cite[Chapter I.4  p. 9]{RS}} \label{fact:beta} Let $k\in \N_{>0}$. Then
$
\beta_{k+1} = - \frac{\beta_k}{\zeta_{k+2}}.
$
\end{fact}

\begin{fact}{\cite[Chapter I.2 p. 4]{RS}} \label{fact:zetaplus1} Let $k\in \N_{>0}$. Then
\[
\zeta_k = a_{k} + \frac{1}{\zeta_{k+1}}.
\]
\end{fact}

\begin{fact}{\cite[Chapter I.2 p. 4]{RS}}\label{fact:zetaexp} Let $k\in \N_{>0}$. Then
\[
a= 
\frac{p_{k}\zeta_{k+1} + p_{k-1}}{q_{k}\zeta_{k+1} + q_{k-1}}.
\]
\end{fact}

\noindent We will now introduce a numeration system due to Ostrowski \cite{Ost}.

\begin{fact}{\cite[Chapter II.4  p. 24]{RS}}\label{ostrowski} Let $N\in \N$. Then $N$ can be written uniquely as
\[
N = \sum_{k=0}^{n} b_{k+1} q_{k},
\]
where $n\in \N$ and the $b_k$'s are in $\N$ such that $b_1<a_1$ and for all $k\in \N_{\leq n}$, $b_k \leq a_{k}$ and, if $b_k = a_{k}$, then $b_{k-1} = 0$.
\end{fact}

\noindent We call the representation of a natural number $N$ given by Fact \ref{ostrowski} the \textbf{Ostrowski representation} of $N$ based on $a$. Of course, we will drop the reference to $a$ whenever $a$ is clear from the context. If $\varphi$ is the golden ratio, the Ostrowski representation based on $\varphi$ is better known as the \textbf{Zeckendorf representation}, see Zeckendorf \cite{Zeckendorf}. We will also need a similar representation of a real number.

\begin{fact}\cite[Chapter II.6  Theorem 1]{RS}
\label{ostrowskireal} Let $c \in \R$ be such that $-\frac{1}{\zeta_1} \leq c < 1-\frac{1}{\zeta_1}$. Then $c$ can be written uniquely in the form
\[
c = \sum_{k=0}^{\infty} b_{k+1} \beta_{k},
\]
where $b_k \in \N$, $0 \leq b_1 < a_1$, $0 \leq b_k \leq a_{k}$, for $k> 1$, and $b_k = 0$ if $b_{k+1} = a_{k+1}$, and $b_k < a_{k}$ for infinitely many odd $k$.
\end{fact}

\subsection*{Square roots of rational numbers} So far, we have only introduced facts about continued fractions that were already used in \cite{H-Twosubgroups}. In order to extend the results from that paper, we will now recall some theorems about continued fractions for square roots of rational numbers.  For the following, fix $d\in \Q_{>0}$ such that $d\neq c^2$ for all $c\in \Q$. When we refer to $p_k,q_k,\beta_k$ and $\zeta_k$, we mean the ones given by the continued fraction expansion of $\sqrt{d}$. In \cite{H-Twosubgroups} we used the fact that the continued fraction expansion of quadratic numbers is periodic. Here we need the following stronger statement for $\sqrt{d}$.

\begin{fact}\label{fact:cfsqrt}\cite[Theorem III.1.5]{RS} The continued fraction expansion of $\sqrt{d}$ is of the form
$[a_0;\overline{a_1,a_2,a_3\dots,a_2,a_1,2a_0}]$, 
where the periodic part without the last term is a palindrome. 
\end{fact}
\noindent Let $m$ be the length of the (minimal) period of the continued fraction expansion of $\sqrt{d}$. 

\begin{fact}\label{eq:zeta} Let $\ell \in \N$. Then 
\[
\zeta_{\ell m + 1} = \frac{1}{\sqrt{d}-a_0}.
\]
\end{fact}
\begin{proof}
By the periodicity and the definition of the $k$-th complete quotient, we obtain
\begin{equation*}
 \zeta_{\ell m+1} = \zeta_1 = \frac{1}{\zeta_0 - a_0} = \frac{1}{\sqrt{d} - a_0}.
\end{equation*}
\end{proof}

\noindent Our proof of Theorem D depends crucially on the following connection between the two sequences $(p_k)_{k\in \N}$ and $(q_k)_{k\in \N}$.

\begin{fact}\label{fact:pq} Let $k\in \N$. Then
\begin{align*}
p_{km} &= a_0 p_{km-1} + d q_{km-1},\\
q_{km} &= a_0 q_{km-1} + p_{km-1}.
\end{align*}
\end{fact}
\begin{proof} By Fact \ref{fact:zetaexp} and Fact \ref{eq:zeta},
\begin{align*}
\sqrt{d} &= \frac{p_{km}\zeta_{km+1} + p_{km-1}}{q_{km}\zeta_{km+1} + q_{km-1}}\\
&= \frac{p_{km}+ \sqrt{d} p_{km-1} - a_0 p_{km-1}}{q_{km} + \sqrt{d} q_{km-1}- a_0q_{km-1}}.
\end{align*}
Hence
\[
\sqrt{d} (q_{km} -a_0 q_{km-1} - p_{km-1}) + d q_{km-1} - p_{km}+ a_0 p_{km-1} =0.
\]
The statement follows from the irrationality of $\sqrt{d}$.
\end{proof}

\begin{fact}\label{fact:pkm}
There exist $s=(s_0,\dots,s_{m-1}), t=(t_0,\dots,t_{m-1})\in (\Q^2)^m$ such that for every $i\in \{0,\dots, m-1\}$ and for every $k\in \N$
\begin{align*}
p_{km} &= s_{i,1} \cdot p_{km+i} + s_{i,2}\cdot p_{km+i-1}\\
p_{km-1} &=  t_{i,1} \cdot p_{km+i} + t_{i,2} \cdot p_{km+i-1}.
\end{align*}
\end{fact}
\begin{proof}
We prove the statement by induction on $i$. When $i=0$, then the statement holds with $s_0=(1,0)$ and $t_0=(0,1)$. Now assume there 
exist $(s_0,\dots,s_i)$ and $(t_0,\dots,t_i)$ in $(\Q^2)^i$ such that for every $k\in \N$
\begin{align*}
p_{km} &= s_{i,1} \cdot p_{km+i} + s_{i,2}\cdot p_{km+i-1}\\
p_{km-1} &=  t_{i,1} \cdot p_{km+i} + t_{i,2} \cdot p_{km+i-1}.
\end{align*}
By the periodicity of the continued fraction expansion of $\sqrt{d}$ and Fact \ref{fact:recursive}, there exists $u \in \N_{> 0}$ such that for every $k \in \N$
\[
p_{km+i+1} = u \cdot p_{km+i} + p_{km+i-1}.
\]
Thus
\begin{align*}
p_{km} &= s_{i,2} p_{km+i+1} + (s_{i,1}-s_{i,2}u) \cdot p_{km+i}\\
p_{km-1} &= t_{i,2} p_{km+i+1} + (t_{i,1}-t_{i,2}u) \cdot p_{km+i}.
\end{align*}
\end{proof}

\begin{cor}\label{cor:pq} There exist $v=(v_0,\dots,v_{m-1}), w=(w_0,\dots,w_{m-1})\in \Q^m$ such that for every $i\in \{0,\dots, m-1\}$ and for every $k\in \N$
\[
q_{km+i} = v_{i} \cdot p_{km+i+1} + w_i \cdot p_{km+i}.
\]
\end{cor}
\begin{proof} 
 By the periodicity of the continued fraction expansion of $\sqrt{d}$ and Fact \ref{fact:recursive}, we have that for $i\in \{0,\dots, m-1\}$ there exist $r_{i,1},r_{i,2}  \in \Z$ such that for every $k \in \N$
\[
q_{km+i} = r_{i,1} \cdot q_{km} + r_{i,2} \cdot q_{km-1}.
\]
By Fact \ref{fact:pq}, 
\begin{align*}
q_{km-1} &= \frac{1}{d} p_{km} - \frac{a_0}{d} p_{km-1},\\
q_{km} &= \frac{a_0}{d} p_{km} + \frac{d-a_0}{d} p_{km-1}.
\end{align*}
Thus for $i\in \{0,\dots, m-1\}$ there exist $u_{i,1},u_{i,2}  \in \Z$ such that for every $k \in \N$
\[
q_{km+i} = u_{i,1} \cdot p_{km} + u_{i,2} \cdot p_{km-1}.
\]
The statement now follows from Fact \ref{fact:pkm}.
\end{proof}
\noindent For purely periodic continued fraction expansions, like the one of the golden ratio $\varphi$, there is an even stronger connection between the $p_k$'s and the $q_k$'s. In that case, $q_{k+1}=p_k$. This fact was used in \cite{H-Twosubgroups} to show the definability of $\lambda_{\varphi}$ in $\Cal R_{\varphi}$. In the next section, we will prove that the weaker statement of Corollary \ref{cor:pq} is enough to establish the definability of $\lambda_{\sqrt{d}}$.

\begin{fact}\label{fact:productzeta} The following equation holds:
\[
\zeta_1 \cdots \zeta_{m+1} = \frac{q_m}{\sqrt{d}-a_0} + q_{m-1}
\]
As a consequence $\zeta_1  \cdots \zeta_{m+1} \in \Q(\sqrt{d})$.
\end{fact}
\begin{proof} Recall that $q_{-1}=0$ and $q_0=1$. Applying Fact \ref{fact:recursive} and Fact \ref{fact:zetaplus1} multiple times, we obtain
\begin{align*}
\zeta_1 \cdots \zeta_{m+1} &= (q_0 \zeta_1 + q_{-1}) \cdot \zeta_2 \cdots \zeta_{m+1}\\
&=\Big( (q_0 (a_1 + \frac{1}{\zeta_2}) + q_{-1}) \zeta_2\Big)
\cdot \zeta_3  \cdots  \zeta_{m+1}\\
&=\big( q_1 \zeta_2 + q_0\big)
\cdot \zeta_3  \cdots \zeta_{m+1}\\
&= \cdots = \big(q_{m-1} \zeta_m + q_{m-2}\big) \cdot \zeta_{m+1}= q_m \zeta_{m+1} + q_{m-1}
\end{align*}
The statements of the fact follows directly from Fact \ref{eq:zeta}.
\end{proof}

\noindent The definability of $\lambda_{\zeta_1 \cdots \zeta_{m+1}}$ in $\Cal S_{\sqrt{d}}$ is a direct consequence. Because of the periodicity of the continued fraction expansion of $\sqrt{d}$ and Fact \ref{fact:beta}, we also get the following fact.
\begin{fact}\label{eq:multbyc} Let $k\in \N$. Then
\[
\zeta_1 \cdots  \zeta_{m+1} \cdot \beta_{k+m} = (-1)^m \cdot \beta_{k}.
\]
\end{fact}
\noindent Hence multiplying a real number $z\in[-\frac{1}{\zeta_1}, 1-\frac{1}{\zeta_1})$ by $\zeta_1 \cdots \zeta_{m+1}$ corresponds to an $m$-shift in the Ostrowski representation of $z$.

\section{Defining scalar multiplication} Let $d \in \Q$. In this section we prove that $\Cal R_{\sqrt{d}}$ defines $\lambda_{\sqrt{d}}$. We can easily reduce to the case that $\sqrt{d}\notin \Q$. Since $\Cal R_{a}$ and $\Cal R_{qa}$ are interdefinable for non-zero $q\in \Q$ and the set of squares of rational numbers is dense in $\R_{\geq 0}$, we can assume that $1.5 < \sqrt{d} < 2$. By Fact \ref{fact:cfsqrt}, the continued fraction expansion of $\sqrt{d}$ is of the form
\[
[a_0;\overline{a_1,a_2,a_3\dots,a_2,a_1,2a_0}].
\]
Denote the length of the period by $m$ and set $s:= \max a_i$. From now on only the structure $\Cal R_{\sqrt{d}}$ is considered. Whenever we say definable, we mean definable in this structure.

\subsection*{Preliminaries} We now recall all the necessary results from Section 4 of \cite{H-Twosubgroups}. The main observation from the section we need is that the structure $(\R,<,+,\Z,\sqrt{d} \Z)$ defines predicates allowing us to definably recover the digits of the Ostrowski representation of a given number. Everything stated here is either explicitly stated in \cite{H-Twosubgroups} or can be obtained by minor modifications.\newline

\noindent Since $1<\sqrt{d}<2$, we have 
\[
[-\frac{1}{\zeta_1}, 1-\frac{1}{\zeta_1})=[1-\sqrt{d},2-\sqrt{d}).\]
We denote this interval by $I$. By the statement after \cite[Definiton 4.1]{H-Twosubgroups}, the set $\{q_k \sqrt{d} : k > 0\}$ is definable. We write $V$ for this set and $s_V$ for the successor function on $V$. The reader can easily verify that $s_V$ is definable since $V$ is definable as well.

\begin{defn} Let $f: \N\sqrt{d} \to \R$ map $n\sqrt{d}$ to $\sum_k b_{k+1} \beta_k$ if $\sum_k b_{k+1} q_k$ is the Ostrowski representation of $n$.
\end{defn}

\noindent By \cite[Lemma 4.3]{H-Twosubgroups} the function $f$ is definable. This allows us to move definably between natural numbers and real numbers whose Ostrowski representations have the same digits.

\begin{defn}
For $i \in \{0,\dots,s\}$ we define $E_i \subseteq V\times I$ such that $(q_{\ell}\sqrt{d},c) \in E_i$ if and only if there is a sequence $(b_i)_{i\in \N_{>0}}$ such that $\sum_{k=0}^{\infty} b_{k+1} \beta_k$ is the Ostrowski representation of $c$ and $b_{\ell+1}= i$.
\end{defn}

\noindent  Lemma 4.11 of \cite{H-Twosubgroups} only states that $E_i$ is definable for $i \in \{0,1\}$. However, the reader can check that its proof can be used easily to conclude that $E_i$ is indeed definable for every $i\in \{0,\dots s\}$.

\subsection*{Defining shifts} We now start to extend the results from \cite{H-Twosubgroups}. In order to define multiplication by $\sqrt{d}$, we have to show that certain generalized shifts in the Ostrowski representation are definable.

\begin{lem} Let $n \in \N_{>1}$ and $j\in\{0,\dots,n-1\}$. Then the set
\[
V_{j,n} := \{ q_l\sqrt{d} \in V \ : \ l = j \mod n \}
\]
is definable.
\end{lem}
\begin{proof} Let $c\in I$ be the unique element of $I$ such that $(q_j\sqrt{d},c)\in E_1$, $(q_l\sqrt{d},c) \in E_0$ for $l<j$ and
\[
\forall q_l\sqrt{d} \in V \ ((q_l\sqrt{d},c) \in E_1) \leftrightarrow \big(\bigwedge_{i=1}^{n-1} (q_{l+i}\sqrt{d},c) \in E_0 \big).
\]
Since $n>1$, such a $c$ exists. Since every element of $V$ and the successor function $s_V$ are definable, so is $c$. It is easy to verify that
\[
V_{j,n} = \{q_l\sqrt{d} \in V \ : \ (q_l\sqrt{d},c) \in E_1\}.
\]
Hence $V_{j,n}$ is definable.
\end{proof}

\noindent Recall that $m$ is defined as the smallest period of the continued fraction of $\sqrt{d}$. Set $t := \max \{ m, 2\}$.

\begin{defn} For $i\in \{0,\dots, t-1\}$, define $B_i\subseteq \N\sqrt{d}$ to be the set of all $n\sqrt{d}$ such that
\[
\forall z \in V \ \big(z \notin V_{i,t} \rightarrow E_0(z,f(n\sqrt{d}))\big) \wedge \big(z \in V_{i,t} \rightarrow (E_0(z,f(n\sqrt{d}))\vee E_1(z,f(n\sqrt{d})))\big).
\]
\end{defn}

\noindent Note that $B_i$ is definable for each $i\in\{0,\dots,t-1\}$. The set $B_i$ contains precisely those of elements of $n\sqrt{d}$ of $\N\sqrt{d}$ for which the digit $b_{k+1}$ of the Ostrowski representation on $n$ is either $0$ or $1$ when $k=i \mod t$, and $0$ otherwise. Indeed, the following lemma follows the definitions of $f$ and $V_{i,t}$.

\begin{lem}\label{lem:propbj} Let $i,n \in \N$ be such that $\sum_{k} b_{k+1} q_k$ is the Ostrowski representation of $n$ and $n\sqrt{d} \in B_i$. Then for every $k\in \N$
\[
b_{k+1} \in \left\{
              \begin{array}{ll}
               \{0,1\}, & \hbox{if $k=i \mod t$;} \\
                \{0\}, & \hbox{otherwise.}
              \end{array}
            \right.
\]
\end{lem}

\noindent Since we chose $t$ to be at least $2$, we obtain the following corollary.

\begin{cor}\label{cor:finitesub} Let $X \subseteq \Cal P(\N)$ be finite. Then there exists $n\sqrt{d} \in B_i$ such that
\[
n = \sum_{k \in X} q_{kt+i}.
\]
\end{cor}
\noindent Hence there is a natural bijection between $B_i$ and the set of finite subsets of $\N$. This is the only place where we need that $t\geq 2$. We now use this observation to define a shift between $B_i$ and $B_j$ when $j=i+1 \mod t$.

\begin{defn} Let $i,j\in \{0,\dots,t-1\}$ be such that $j = i+1 \mod t$. Let $S_i : B_i \to B_j$ map $x \in B_i$ to the unique $y \in B_j$ such that
\[
E_0(1,y) \wedge \forall z\in V (E_{1}(z,x) \leftrightarrow E_{1}(s_V(z),y)).
\]
\end{defn}

\noindent It follows from Corollary \ref{cor:finitesub} that the unique $y \in B_j$ in the above Definition always exists. Since $B_i$ and $E_1$ are definable, so is $S_i$. Moreover, note that the function $S_i$ is simply a shift by one in the Ostrowski representation. The following lemma makes this statement precise.

\begin{lem}\label{lem:oshift} Let $i,j\in \{0,\dots,t-1\}$ be such that $j = i+1 \mod t$. Let $n\sqrt{d} \in B_i$ and $\ell\in \N$ such that $S_i(n\sqrt{d})=\ell\sqrt{d}$ and $\sum_k b_{k+1} q_k$ is the Ostrowski representation of $n$. Then the Ostrowski representation of $\ell$ is $\sum_k b_{k+1} q_{k+1}$.
\end{lem}

\noindent It is worth pointing out that the sum $\sum_k b_{k+1} q_{k+1}$ is only the Ostrowski representation of $\ell$, because being in $B_i$ implies that all $b_{k}$ are in $\{0,1\}$, and that whenever $b_k=1$, then $b_{k-1}=0$. In general, when we take an Ostrowski representation and shift it as in Lemma \ref{lem:oshift}, there is no guarantee that the resulting sum is again an Ostrowski representation. However, in order to define multiplication by $\sqrt{d}$, we will have to make shifts that may result in sums that are not Ostrowski representations. Towards that goal, we will now introduce a new definable object $C$ which in a way made precise later, contains all Ostrowski representations and is closed under shifts.

\begin{defn} For $\ell\in \{0,\dots,t-1\}$, define
\[
C_{\ell} := \{ (x_1,\dots, x_s) \in B_{\ell}^s \ : \ \bigwedge_{1\leq i<j\leq s} \forall z \in V \big (E_1(z,f(x_j)) \rightarrow E_1(z,f(x_i))\big)\}.
\]
Set $C := C_0\times \dots \times C_{t-1}$. Define $T : V \times C \to \{0,\dots,s\}$ by
\[
(z,(c_0,\dots,c_{t-1})) \mapsto \max \{ j \ : \bigvee_{i=0}^{t-1}  E_1(z,f(c_{i,j})) \wedge z \in V_{i,t} \} \cup \{0\}.
\]
\end{defn}

\noindent In the following, we will often work with an element $c=(c_0,\dots,c_{t-1}) \in C$, where $c_i$ is assumed to be in $C_i$. When we refer to $c_{i,j}$, as is done in the definition of $T$, we will always mean the $j$-th component of $c_i$. Note that for every $z\in V$ there exists a unique $i\in \{0,\dots,t-1\}$ such that $z \in V_{i,t}$. Hence for that $i$, we immediately get from the definition of $T$ that for every $c\in C$
\[
T(z,c)= \max \{ j \ :  E_1(z,f(c_{i,j})) \wedge z \in V_{i,t} \}\cup \{0\}.
\]
Thus the conjunction in the definition of $T$ can be dropped if $i$ is assumed to satisfy $z \in V_{i,t}$.

\begin{lem}\label{lem:isoC} Let $\alpha : \N \to \{0,\dots,s\}$ be a function that is eventually zero. Then there is a unique $c \in C$ such that $T(q_{\ell}\sqrt{d},c) = \alpha(\ell)$ for all $\ell \in \N$.
\end{lem}
\begin{proof} By Corollary \ref{cor:finitesub} we can find for each $j\in \{0,\dots,t-1\}$ and for each finite $X \in \Cal P(\N)$ an element $n\sqrt{d}\in B_j$ such that $k\in X$ if and only if $E_1(q_{kt+j+1}\sqrt{d},f(n\sqrt{d}))$. The statement of the Lemma follows easily.\end{proof}

\noindent As a corollary of Lemma \ref{lem:isoC} we get that the set of Ostrowski representations can be embedded into $C$.

\begin{cor}\label{cor:T} Let $n\in \N$ and $\sum_k b_{k+1} q_k$ be the Ostrowski representation of $n$. Then there is a unique $c \in C$ such that $b_{k+1} = T(q_k\sqrt{d},c)$ for all $k\in \N$.
\end{cor}

\begin{defn} Let $R: \N\sqrt{d} \to C$ map $n\sqrt{d}$ to the unique $c \in C$ such that
\[
\bigwedge_{i=1}^s \ \forall z \in V \ E_i(z,f(n\sqrt{d})) \leftrightarrow T(z,c)=i.
\]
\end{defn}
\noindent By Corollary \ref{cor:T} the unique $c$ in the preceding definition indeed exists. Note that $R$ is definable. The motivation for the definition of $C$ was to be able to define shifts.

\begin{defn} Let $S : C \to C$ be given by
\[
(c_0,\dots,c_{t-1}) \mapsto (S_{t-1}(c_{t-1}),S_0(c_0),\dots,S_{t-2}(c_{t-2})).
\]
\end{defn}

\noindent For $\ell \geq 1$ we denote the $\ell$-th compositional iterate of $S$ by $S^{\ell}$. The following Lemma shows that the function $S$ is indeed a shift operation with respect to $T$.

\begin{lem}\label{lem:STconnection} Let $c \in C$ and $k \in \N$. Then $T(q_k\sqrt{d},c) = T(q_{k+1}\sqrt{d},S(c)).$
\end{lem}
\begin{proof} Let $c=(c_0,\dots,c_{t-1})\in C$. Let $i,\ell \in \{0,\dots,t-1\}$ such that $k=i\mod t$ and $\ell=i+1\mod t$. By Definition of $S_i$, we have that for $j\in \{1,\dots,s\}$
\[
E_1(q_k\sqrt{d},c_{i,j}) \leftrightarrow E_1(q_{k+1}\sqrt{d},S_i(c_{i,j})).
\]
Since $q_k \sqrt{d} \in V_{i,t}$ and $q_{k+1} \sqrt{d} \in V_{\ell,t}$, it follows immediately from the definition of $T$ that
\begin{align*}
T(q_k\sqrt{d},c) &=  \max \{ j \ : \ E_1(q_k\sqrt{d},c_{i,j}) \} \cup \{0\}\\
& = \max \{ j \ : \ E_1(q_{k+1}\sqrt{d},S(c_{i,j})) \} \cup \{0\}= T(q_{k+1}\sqrt{d},S(c)).
\end{align*}
\end{proof}

\noindent After showing that $\N\sqrt{d}$ can be embedded into $C$ and that there exists a definable shift operation, the next step is to recover natural numbers and real numbers from $C$. To achieve this, we define the following two functions.

\begin{defn} For $u=(u_0,\dots,u_{t-1})\in \Q^t$, let $\Sigma_u :C \to \R$ be defined by
\[
(c_0,\dots,c_{t-1}) \mapsto \sum_{i=0}^{t-1} u_i \sum_{j=1}^s c_{i,j},
\]
and $F_u: C \to \R$ by
\[
(c_0,\dots,c_{t-1}) \mapsto \sum_{i=0}^{t-1} u_i \sum_{j=1}^s f(c_{i,j}).
\]
\end{defn}

\noindent As is made precise in the following Proposition, one should think of the image of $C$ under $\Sigma_u$ and $F_u$ as the set of numbers that can be expressed (not necessarily uniquely) in some generalized Ostrowski representation.






\begin{prop}\label{prop:gshift} Let $u=(u_0,\dots,u_{t-1}) \in \Q^t$ and $n\in \N$ be such that $\sum_k b_{k+1} q_k$ is the Ostrowski representation of $n$. Then
\begin{align*}
\Sigma_u(S^{\ell}(R(n\sqrt{d}))) &= \sum_{i=0}^{t-1} u_i \sum_{k=0}^{\infty} b_{kt+i+1} q_{kt+i+\ell}\sqrt{d}, \hbox{ and }\\
F_u(S^{\ell}(R(n\sqrt{d}))) &= \sum_{i=0}^{t-1} u_i \sum_{k=0}^{\infty} b_{kt+i+1} \beta_{kt+i+\ell}.
\end{align*}
\end{prop}
\begin{proof} By Corollary \ref{cor:T}, we have that $b_{k+1} = T(q_k\sqrt{d},R(n\sqrt{d}))$, for all $k \in \N$.
By Lemma \ref{lem:STconnection}, $T(q_k\sqrt{d},R(n\sqrt{d})) = T(q_{k+\ell}\sqrt{d},S^\ell(R(n\sqrt{d})))$ for all $k\in \N$. For ease of notation, denote $S^\ell(R(n\sqrt{d}))$ by $c=(c_0,\dots,c_{t-1})$.
Then we have for each $k\in \N$, $i\in \{0,\dots,t-1\}$ and $j\in \{1,\dots,s\}$ that
\[E_1(q_{kt+i+\ell}\sqrt{d},c_{i,j}) \hbox{ if and only if } b_{kt+i+1} \leq j.\]
Hence
\begin{align*}
 \sum_{j=1}^{s} c_{i,j} =\sum_{j=1}^{s} \sum_{k} |\{ j :  E_1(q_{kt+i}\sqrt{d},c_{i,j})\}| q_{kt+i+\ell}\sqrt{d}=  \sum_{k} b_{kt+i+1} q_{kt+i+\ell}\sqrt{d}.
\end{align*}
With the same argument, the reader can check that
\[
\sum_{j=1}^{s} f(c_{i,j})=\sum_{k} b_{kt+i+1} \beta_{kt+i+\ell}.
\]
We can easily deduce the statement of the Lemma from the definitions of $\Sigma$ and $F$.
\end{proof}

\subsection*{Proof of Theorem D} In this subsection, we will give a proof of Theorem D. When we say that for a real number $b\in \R$ and a subset $X$ the restriction of $\lambda_b$ to $X$ is definable, we just mean that the graph of the restriction $\lambda_b|_X$ is definable.\newline

\noindent Here is an outline how we proceed to prove Theorem D: we first combine Fact \ref{eq:multbyc} and Corollary \ref{cor:pq} with technology developed in the previous subsection, to show that the restrictions of $\lambda_{\sqrt{d}}$ to $\N$ and to $f(\N\sqrt{d})$ are definable. Using arguments from \cite{H-Twosubgroups} we conclude that $\lambda_{\sqrt{d}}$ is definable.

\begin{lem}\label{lem:restoI} Let $n\in \N$. Then
\[
f(n\sqrt{d}) = (-1)^m(\frac{q_m}{\sqrt{d}-a_0} +q_{m-1}) \cdot F_{(1,\dots,1)}(S^m(R(n\sqrt{d}))).
\]
\end{lem}
\begin{proof} Let $\sum_k b_{k+1} q_k$ be the Ostrowski representation of $n$. By  Fact \ref{eq:multbyc}, Fact \ref{fact:productzeta} and Proposition \ref{prop:gshift}
\begin{align*}
f(n\sqrt{d}) &= \sum_k b_{k+1} \beta_k = (-1)^m(\frac{q_m}{\sqrt{d}-a_0} +q_{m-1})  \sum_{k} b_{k+1} \beta_{k+m}\\
&=(-1)^m(\frac{q_m}{\sqrt{d}-a_0} +q_{m-1})  \cdot F_{(1,\dots,1)}(S^m(R(n\sqrt{d}))).
\end{align*}
\end{proof}

\begin{cor}\label{cor:restoI} The restriction of $\lambda_{\sqrt{d}}$ to $f(\N\sqrt{d})$ is definable.
\end{cor}
\begin{proof} 
Let $a,b \in \Q$ be such that
\[
\frac{q_m}{\sqrt{d}-a_0} +q_{m-1} = a \sqrt{d} + b.
\]
By Lemma \ref{lem:restoI} and the injectivity of $f$, the restriction of $\lambda_{(a\sqrt{d}+b)^{-1}}$ to $f(\N\sqrt{d})$ is definable.
Since $(a\sqrt{d}+b)^{-1} = \frac{a\sqrt{d}-b}{a^2d-b^2}$, we have
\[
\lambda_{\sqrt{d}}(x) = \frac{a^2d-b^2}{a}\lambda_{(a\sqrt{d}+b)^{-1}}(x) + \frac{b}{a}x.
\]
Hence the restriction of $\lambda_{\sqrt{d}}$ to $f(\N\sqrt{d})$ is definable.
\end{proof}

\begin{lem}\label{lem:restoN} There are $v,w \in \Q^t$ such that for every  $n \in \N$
\[
n =  \Sigma_{v}(S(R(n\sqrt{d}))) - F_{v}(S(R(n\sqrt{d}))) +  \Sigma_{w}(R(n\sqrt{d})) - F_{w}(R(n\sqrt{d})).
\]
\end{lem}
\begin{proof} Let $v,w \in \Q^t$ be given by Corollary \ref{cor:pq}. Note that $p_k = q_k\sqrt{d} - \beta_k$.
By Proposition \ref{prop:gshift}
\begin{align*}
n &= \sum_{k=0}^{\infty} b_{k+1} q_k = \sum_{i=0}^{t-1} \sum_{k=0}^{\infty} b_{kt+i+1} q_{kt+i}\\
&=  \sum_{i=0}^{t-1} \sum_{k=0}^{\infty} b_{kt+i+1} (v_i \cdot p_{kt+i+1} + w_i \cdot p_{kt+i})\\
&=  \sum_{i=0}^{t-1} \sum_{k=0}^{\infty} b_{kt+i+1} \Big(v_i (q_{kt+i+1}\sqrt{d} - \beta_{kt+i+1}) + w_i (q_{kt+i+1}\sqrt{d} - \beta_{kt+i+1})\Big)\\
&=  \Sigma_{v}(S(R(n\sqrt{d}))) - F_{v}(S(R(n\sqrt{d}))) +  \Sigma_{w}(R(n\sqrt{d})) - F_{w}(R(n\sqrt{d})).
\end{align*}
\end{proof}

\begin{cor}\label{cor:restoN} The restriction of $\lambda_{\sqrt{d}}$ to $\N$ is definable.
\end{cor}
\begin{proof} By Lemma \ref{lem:restoN} the restriction of $\lambda_{\sqrt{d}^{-1}}$ to $\N\sqrt{d}$ is definable. Since $\lambda_{\sqrt{d}}$ is the inverse function of $\lambda_{\sqrt{d}^{-1}}$ and $\lambda_{\sqrt{d}^{-1}}(\sqrt{d}\N)=\N$, it follows that the restriction of $\lambda_{\sqrt{d}}$ to $\N$ is definable.
\end{proof}

\begin{proof}[Proof of Theorem D] Here we follow the argument in the proof of \cite[Theorem 5.5]{H-Twosubgroups}. First note that it is enough to define $\lambda_{\sqrt{d}}$ on $\R_{\geq 0}$. Let $Q : \N + f(\N\sqrt{d}) \to \R$ map $m + f(n\sqrt{d})$ to $\lambda_{\sqrt{d}}(m) + \lambda_{\sqrt{d}}(f(n\sqrt{d}))$. It is immediate that $Q$ is well-defined and that $Q$ and $\lambda_{\sqrt{d}}$ agree on the domain of $Q$. By Corollary \ref{cor:restoI} and Corollary \ref{cor:restoN}, $Q$ is definable.
Since $\N + f(\N\sqrt{d})$ is dense in $[1-\sqrt{d},\infty)$ and multiplication by $\sqrt{d}$ is continuous, the graph of $\lambda_{\sqrt{d}}$ on $[1-\sqrt{d},\infty)$ is the topological closure of the graph of $Q$ in $\R^2$.
Thus the restriction of $\lambda_{\sqrt{d}}$ to $\R_{\geq 0}$ is definable.
\end{proof}

\section{Conclusion}

This paper solves the question left open in \cite{H-Twosubgroups} whether the theory of $(\R,<,+,\Z,\lambda_a)$ is decidable whenever $a$ is a quadratic irrational number. We achieve this by showing that $(\R,<,+,\Z,\Z\sqrt{d})$ defines $\lambda_{\sqrt{d}}$ whenever $d\in \Q$. Since the theory of $(\R,<,+,\Z,\Z\sqrt{d})$ is known to be decidable by \cite[Theorem A]{H-Twosubgroups}, we can conclude that the theory of $(\R,<,+,\Z,\lambda_a)$ is indeed decidable when $a$ is quadratic.\newline

\noindent We finish with a few remarks about related results and open questions.

\subsection*{1} We do not know whether Theorem D holds when $\sqrt{d}$ is replaced by an arbitrary real number $a$, even in the case when $a$ is quadratic. By \cite[Theorem A]{H-Twosubgroups} we know for quadratic $a$ that the theory of $(\R,<,+,\Z,\Z a)$ is decidable. However, when $a$ is non-quadratic not much is known.

\subsection*{2} Let $a \in \R\setminus \Q$. Let $x^{a} : \R \to \R$ map $t$ to $t^a$ if $t>0$ and to $0$ otherwise. An isomorphic copy of $\Cal S_a$ is definable in the structure $(\R,<,+,\cdot,2^{\Z},x^{a})$. But by \cite[Theorem 1.3]{discrete} the latter structure defines $\Z$ and hence its theory is undecidable, even if $a$ is quadratic.

\subsection*{3} Questions considered in this paper can also be asked for $\Q$ instead of $\Z$. By Robinson \cite{julia} the structure $(\R,<,+,\cdot,\Q)$ defines $\Z$ and therefore its theory is undecidable. On the other hand, $(\R,<,+,\Q)$ is modeltheoretically very well behaved, see van den Dries \cite{densepairs}, and its theory is decidable. So here we can also ask how many traces of multiplication can be added to the latter structure without destroying its tameness? By recent work of Block Gorman, Hieronymi and Kaplan in \cite{BHK}, $(\R,<,+,\Q,\lambda_a)$ is model-theoretically tame for every $a\in \R$. Furthermore, the theory of  $(\R,<,+,\Q,\lambda_a)$  is decidable as long as $\Q(a)$ has a computable presentation as an ordered field, and the question whether a finite subset of $\Q(a)$ is $\Q$-linearly independent is decidable.

  \bibliographystyle{plain}
  \bibliography{hieronymi}

\begin{thebibliography}{10}

\bibitem{BHK}
Alexi Block~Gorman, Philipp Hieronymi, and Elliot Kaplan.
\newblock Pairs of theories satisfying a {M}ordell-{L}ang condition.
\newblock {\em arXiv:1806.00030}, 2018.

\bibitem{BJW}
Bernard Boigelot, S{\'e}bastien Jodogne, and Pierre Wolper.
\newblock An effective decision procedure for linear arithmetic over the
  integers and reals.
\newblock {\em ACM Trans. Comput. Log.}, 6(3):614--633, 2005.

\bibitem{BRW}
Bernard Boigelot, St{\'e}phane Rassart, and Pierre Wolper.
\newblock On the expressiveness of real and integer arithmetic automata
  (extended abstract).
\newblock In {\em Proceedings of the 25th International Colloquium on Automata,
  Languages and Programming}, ICALP '98, pages 152--163, London, UK, 1998.
  Springer-Verlag.

\bibitem{Buchi}
J.~Richard B{\"u}chi.
\newblock On a decision method in restricted second order arithmetic.
\newblock In {\em Logic, {M}ethodology and {P}hilosophy of {S}cience ({P}roc.
  1960 {I}nternat. {C}ongr .)}, pages 1--11. Stanford Univ. Press, Stanford,
  Calif., 1962.

\bibitem{Goedel}
Kurt G{\"o}del.
\newblock \"{U}ber formal unentscheidbare {S}\"atze der {P}rincipia
  {M}athematica und verwandter {S}ysteme {I}.
\newblock {\em Monatsh. Math. Phys.}, 38(1):173--198, 1931.

\bibitem{discrete}
Philipp Hieronymi.
\newblock Defining the set of integers in expansions of the real field by a
  closed discrete set.
\newblock {\em Proc. Amer. Math. Soc.}, 138(6):2163--2168, 2010.

\bibitem{H-Twosubgroups}
Philipp Hieronymi.
\newblock Expansions of the ordered additive group of real numbers by two
  discrete subgroups.
\newblock {\em J. Symbolic Logic}, to appear, arXiv:1407.7002, 2015.

\bibitem{HNP}
Philipp Hieronymi, Danny Nguyen, and Igor Pak.
\newblock Presburger arithmetic with algebraic scalar multiplication.
\newblock {\em arXiv:1805.03624}, 2018.

\bibitem{HT}
Philipp Hieronymi and Michael Tychonievich.
\newblock Interpreting the projective hierarchy in expansions of the real line.
\newblock {\em Proc. Amer. Math. Soc.}, 142(9):3259--3267, 2014.

\bibitem{ivp}
Chris Miller.
\newblock Expansions of dense linear orders with the intermediate value
  property.
\newblock {\em J. Symbolic Logic}, 66(4):1783--1790, 2001.

\bibitem{Ost}
Alexander Ostrowski.
\newblock Bemerkungen zur {T}heorie der {D}iophantischen {A}pproximationen.
\newblock {\em Abh. Math. Sem. Univ. Hamburg}, 1(1):77--98, 1922.

\bibitem{julia}
Julia Robinson.
\newblock Definability and decision problems in arithmetic.
\newblock {\em J. Symbolic Logic}, 14:98--114, 1949.

\bibitem{RS}
Andrew~M. Rockett and Peter Sz{\"u}sz.
\newblock {\em Continued fractions}.
\newblock World Scientific Publishing Co., Inc., River Edge, NJ, 1992.

\bibitem{skolem}
Thoralf Skolem.
\newblock {\"U}ber einige {S}atzfunktionen in der {A}rithmetik.
\newblock {\em Skr. Norske Vidensk. Akad., Oslo, Math.-naturwiss. Kl.},
  7:1--28, 1931.

\bibitem{smor}
Craig Smory{\'n}ski.
\newblock {\em Logical number theory. {I}}.
\newblock Universitext. Springer-Verlag, Berlin, 1991.
\newblock An introduction.

\bibitem{densepairs}
Lou van~den Dries.
\newblock Dense pairs of o-minimal structures.
\newblock {\em Fund. Math.}, 157(1):61--78, 1998.

\bibitem{Vardi}
Moshe~Y. Vardi.
\newblock The {B}\"uchi complementation saga.
\newblock In {\em S{TACS} 2007}, volume 4393 of {\em Lecture Notes in Comput.
  Sci.}, pages 12--22. Springer, Berlin, 2007.

\bibitem{weis}
Volker Weispfenning.
\newblock Mixed real-integer linear quantifier elimination.
\newblock In {\em Proceedings of the 1999 {I}nternational {S}ymposium on
  {S}ymbolic and {A}lgebraic {C}omputation ({V}ancouver, {BC})}, pages 129--136
  (electronic). ACM, New York, 1999.

\bibitem{Zeckendorf}
E.~Zeckendorf.
\newblock Repr\'esentation des nombres naturels par une somme de nombres de
  {F}ibonacci ou de nombres de {L}ucas.
\newblock {\em Bull. Soc. Roy. Sci. Li\`ege}, 41:179--182, 1972.

\end{thebibliography}

\end{document}